\documentclass[12pt]{amsart}
\usepackage{epsfig,color}

\makeatother

\theoremstyle{definition}

\def\fnum{equation} 
\newtheorem{Thm}[\fnum]{Theorem}
\newtheorem{Cor}[\fnum]{Corollary}

\newtheorem{Lem}[\fnum]{Lemma}

\newtheorem{Pro}[\fnum]{Proposition}

\numberwithin{equation}{section}
\newcommand{\Vol}{{\text{Vol}}}

\newcommand{\nn}{{\bf{n}}}

\newcommand{\Sing}{{\text{Sing}}}
\newcommand{\Reg}{{\text{Reg}}}

\newcommand{\cov}{{\text{Cov}}}

\def\RR{{\bold R}}

\def\SS{{\bold S}}

\newcommand{\e}{{\text {e}}}

\newcommand{\cF}{{\mathcal{F}}}
\newcommand{\cL}{{\mathcal{L}}}

\newcommand{\eqr}[1]{(\ref{#1})}

\def\bH{{\bold H}}

\title{A strong Frankel theorem for shrinkers }

\author{Tobias Holck Colding}%
\address{MIT, Dept. of Math.\\
77 Massachusetts Avenue, Cambridge, MA 02139-4307.}
\author{William P. Minicozzi II}%

\thanks{The  authors
were partially supported by NSF  DMS Grants   2104349  and 2005345.}

\email{colding@math.mit.edu and minicozz@math.mit.edu}

\begin{document}

\maketitle

\begin{abstract}
We prove a strong Frankel theorem for mean curvature flow shrinkers in all dimensions: 
Any two shrinkers in a sufficiently large  ball  must intersect.
  In particular, the shrinker itself must be connected in all large balls.  The key to the proof is a strong Bernstein theorem for incomplete stable Gaussian surfaces.
 \end{abstract}
 
% \tableofcontents

\section{Introduction}

A shrinker $\Sigma^n \subset \RR^N$ is a submanifold so that mean curvature flow (MCF) starting from $\Sigma$  
evolves by scaling and shrinks homothetically  to become singular.   The blow  up at any  MCF singularity is a shrinker;  understanding  them is crucial to understanding  singularity formation,   \cite{H, I, W}.
Shrinkers also have a variational interpretation as Gaussian minimal submanifolds.

\vskip1mm
We will prove a strong Frankel theorem for shrinkers that holds even with  boundary:

\begin{Thm}	\label{c:connecting}
If  $\Sigma_1^n , \Sigma_2^n  \subset \overline{B_R} \subset  \RR^{n+1}$ are   compact
embedded shrinkers,   $\partial \Sigma_1 , \partial \Sigma_2 \subset \partial B_R$,  $R \geq R_n$, and $B_R \cap \Sigma_i \ne \emptyset$, 
%\begin{align}
%	B_{R} \cap \partial \Sigma_i = \emptyset {\text{ and }} B_{R} \cap \Sigma_i \ne \emptyset \, , 
%\end{align}
then 
\begin{align}
	B_{R_n} \cap \Sigma_1 \cap \Sigma_2 \ne \emptyset \, .
\end{align}
\end{Thm}

The constant $R_n$ only depends on the dimension $n$.

\vskip1mm
A classical theorem of Hadamard from 1897 proves that any two closed geodesics in  a complete surface with positive curvature must intersect.
In \cite{F}, Frankel generalized this to   compact minimal hypersurfaces in a manifold of positive Ricci curvature. 
Frankel called this the  ``Generalized Hadamard Theorem,'' but it is now known as the Frankel theorem.  It has 
played an important role in minimal surface theory,  \cite{L},   free boundary problems, \cite{FL}, and mean
curvature flow, \cite{M1, M2, MW}.

\vskip1mm
A  consequence of Theorem \ref{c:connecting} is that shrinkers are connected in all sufficiently large balls around the origin:

\begin{Cor}	\label{c:connectup}
If $\Sigma^n   \subset \RR^{n+1}$ is a compact 
embedded shrinker  with $R \geq R_n$ and 
\begin{align}
	\partial \Sigma  \subset \partial B_{R} {\text{ and }} B_{R} \cap \Sigma \ne \emptyset \, , 
\end{align}
then $B_{R} \cap \Sigma$ is connected.  In particular, if $\Sigma$ is complete and $\partial \Sigma = \emptyset$, then $B_{R} \cap \Sigma$ is connected for every $ 
R \geq R_n$.
\end{Cor}

 The simplest shrinkers are  generalized cylinders   $\SS^k_{\sqrt{2k}} \times \RR^{n-k}$, where the $k$-sphere is centered at  $0$ with radius $\sqrt{2k}$ and  $ \RR^{n-k}$ contains $0$.   Here $k=0, 1 , \dots ,n$, where $k=0$ gives a hyperplane, $k=n$ gives an $n$-sphere, 
 and we allow all rotations of the Euclidean factor.
 Taking one of the shrinkers in 
  Theorem \ref{c:connecting} to be a generalized cylinder implies that
 any  shrinker must intersect   such a cylinder:

\begin{Cor}	\label{c:halfspace}
If $\Sigma^n   \subset \RR^{n+1}$ is a properly
embedded shrinker, $R \geq R_n$, 
\begin{align}
	\partial \Sigma  \subset \partial B_{R} {\text{ and }} B_{R} \cap \Sigma \ne \emptyset \, , 
\end{align}
then $B_{R_n} \cap \Sigma$ intersects every $\SS^k_{\sqrt{2k}} \times \RR^{n-k}$ for $k=0,1, \dots , n$.
\end{Cor}

There is also a corresponding rigidity theorem (see Theorem \ref{t:halfspace2} for a stronger version):

\begin{Thm}	\label{t:halfspace}
 If $\Sigma^n   \subset \RR^{n+1}$ is a complete properly embedded shrinker that lies on one side of 
 $\SS^k_{\sqrt{2k}} \times \RR^{n-k}$ for some $k$, then $\Sigma = \SS^k_{\sqrt{2k}} \times \RR^{n-k}$.
 \end{Thm}

The case $k=0$ of Theorem \ref{t:halfspace} is a strong halfspace theorem for shrinkers.
The strong halfspace theorem for minimal surfaces, \cite{HM}, states that any complete  proper minimal surface in $\RR^3$ that lies in a half-space must itself be a plane.  Unlike the results here, the half-space theorem for minimal surfaces holds only in $\RR^3$ and is a global result where it is essential that the   minimal surface is complete.  However, for embedded minimal disks, there is a  local estimate
that plays a key role in understanding  embedded minimal surfaces,  \cite{CM3}.

 Theorem $3$ in \cite{PR} gave the first half-space type-theorem for complete proper shrinkers without boundary, cf. \cite{CE}.  Corollary \ref{c:halfspace} only requires a large enough piece of a shrinker and   it holds for all generalized cylinders.
 
 The proof of the classical Frankel theorem, \cite{F}, used the second variation for geodesics.  This argument was extended to closed shrinkers, and more generally weighted closed minimal surfaces, 
in \cite{WW}.
 The paper
 \cite{IPR} proved a Frankel theorem for complete noncompact shrinkers without boundary under some assumptions at infinity, this time using a Reilly formula and potential theory; see also \cite{BW, W2}.
  
 The approach here is different, relying on a barrier argument and strong Bernstein theorem for Gaussian stable  shrinkers, possibly with boundary. It follows from \cite{CM1} that every properly embedded complete Gaussian minimal hypersurface $\Sigma^n \subset \RR^{n+1}$ must be Gaussian unstable
(cf. Theorem $0.5$ in \cite{CM4}).  The next theorem gives a strong generalization of this, allowing for boundary outside of a fixed ball
(see Theorem \ref{t:stable2} for a stronger result).

\begin{Thm}	\label{t:stable}
If $\Sigma^n \subset   \RR^N$ is a properly embedded shrinker, $R \geq R_n$,   $B_{R}\cap \Sigma \ne \emptyset$
and   $ B_{R} \cap \partial \Sigma = \emptyset$, then $\Sigma$ is Gaussian unstable.
\end{Thm}

   There are three contrasts with the classical Euclidean Bernstein theorems.  First, we do not require completeness - instead, it is already possible to rule out any sufficient large piece.  Second, the theorem applies in all dimensions, instead of holding just up to $\RR^7$.
   Finally, the theorem applies in all codimension, while the classical Bernstein theorems do not hold already in codimension two since 
   complex submanifolds are area-minimizing, \cite{Fe}.

\section{Gaussian area}

We will next recall the variational characterization of shrinkers as Gaussian minimal submanifolds.  It will be useful to work with more general 
rectifiable varifolds, rather than restricting to smooth submanifolds. 
As in \cite{CM1}, define the Gaussian area, or $F$ functional, of an $n$-dimensional rectifiable varifold $\Sigma^n \subset \RR^N$ by
\begin{align}
	F(\Sigma) = \left(4 \, \pi \right)^{ - \frac{n}{2} } \, \int_{\Sigma} \e^{ - \frac{|x|^2}{4} } \, .
\end{align}
The constant $\left(4 \, \pi \right)^{ - \frac{n}{2} } $ makes the Gaussian area of an $n$-plane through the origin equal to one.  Critical points of $F$ are said to be $F$-stationary (see section $12$ in \cite{CM1}).  If $\Sigma$ is an immersed submanifold, then $F$-stationary is equivalent to the shrinker equation
\begin{align}	\label{e:seq}
	\bH = \frac{1}{2} \, x^{\perp} \, , 
\end{align}
where $\bH$ is the mean curvature vector and $x^{\perp}$ is the perpendicular part of the position vector field.{\footnote{See
Proposition $3.6$ in \cite{CM1} for hypersurfaces, cf. \cite{AHW,AS,LL}.}} More generally, $F$-stationary varifolds are weak solutions of \eqr{e:seq}.

The simplest shrinkers are an $n$-plane through the origin, where both $\bH$ and $x^{\perp}$ vanish identically, and an $n$-sphere
of radius $\sqrt{2n}$ centered at the origin.  These are the extreme cases of the generalized cylinders $\SS^{k}_{\sqrt{2k}} \times \RR^{n-k}$, for $k=0, \dots , n$.

\subsection{Volume bounds for $F$-stationary rectifiable varifolds}

The next proposition gives local bounds on volume growth for shrinkers; the argument is  adapted from global results \cite{CZ, CaZ}   for complete shrinkers.

\begin{Pro}	\label{l:areagrow}
If $\Sigma^n \subset B_{\bar{R}} \subset \RR^N$ is a proper $F$-stationary rectifiable varifold and $\sqrt{4+2n} \leq r_1 < r_2 < \bar{R}$, then
\begin{align}	\label{e:evg}
	\left( 1 - \frac{2n}{r_2^2} \right) \, \frac{\Vol (B_{r_2} \cap \Sigma)}{r_2^n} & \leq \frac{\Vol (B_{r_1} \cap \Sigma)}{r_1^n}  \, , \\
	\int_{B_{r} \cap \Sigma} |\bH|^2 &\leq \frac{n}{2} \, 
	\Vol (B_{r} \cap \Sigma) {\text{ at regular values of }} |x| \, .	\label{e:H2bound}
\end{align}
\end{Pro}

\begin{proof}
Recall that $\nabla^T |x| = \frac{|x^T|}{|x|}$. 
Define  $V(r)$ to be the volume of $B_r \cap \Sigma$ and $T(r) = \int_{B_r \cap \Sigma} |\bH|^2$.  The co-area formula ($3.2.22$ in \cite{Fe}) gives that $V$ and $T$ are absolutely continuous functions where the derivatives are given at regular values of $|x|$ (almost everywhere) by 
\begin{align}
	V' (r) &= \int_{\partial B_r \cap \Sigma} \frac{1}{|\nabla^T \, |x||} =  r\, \int_{\partial B_r \cap \Sigma} \frac{1}{|x^T|} \, , \\
	T'(r) &= r\, \int_{\partial B_r \cap \Sigma} \frac{|\bH|^2}{|x^T|} \, .
\end{align}
On the other hand, the shrinker equation gives that $\Delta_{\Sigma} \, |x|^2 = 2\,n - 4\, |\bH|^2$, so the divergence theorem gives
at regular values of $|x|$ that
\begin{align}
	2n\, V(r) - 4\, T(r) &= \int_{B_r \cap \Sigma} \Delta_{\Sigma} \, |x|^2 = 2 \, \int_{\partial B_r \cap \Sigma} |x^T| =  2 \, \int_{\partial B_r \cap \Sigma}\frac{ |x^T|^2}{|x^T|} \notag \\
	&= 2 \, \int_{\partial B_r \cap \Sigma}\frac{|x|^2 -4\, |\bH|^2}{|x^T|} = 2\, r \, V'(r) - \frac{8}{r} \, T'(r) \, .
\end{align}
One consequence of the first two equalities is that $T(r) \leq \frac{n}{2} \, V(r)$, giving the second claim \eqr{e:H2bound}.
Using the full string of equalities gives 
\begin{align}
	r \, V' - n \, V = \frac{4}{r} \, T' -2 \, T \, .
\end{align}
In particular, we get that
\begin{align}
	\left( r^{-n} \, V \right)' =  r^{-n-1} \left( r\, V' - n \, V \right)  = r^{-n-1} \left( \frac{4}{r} \, T' - 2 \, T    \right) \, .
\end{align}
Integrating this from $r_1$ to $r_2$ and then integrating by parts on the right gives
\begin{align}
	\frac{V(r_2)}{r_2^n} - \frac{V(r_1)}{r_1^n} &= \int_{r_1}^{r_2} \left( 4 \, r^{-n-2} \, T' -  2\, r^{-n-1}  \, T \right) \, dr \notag \\
	&= 4\, \frac{T(r_2)}{r_2^{n+2}}  - 4\, \frac{T(r_1)}{r_1^{n+2}}  + \int_{r_1}^{r_2} \left( 4 \, (n+2) \, r^{-n-3}   -  2\, r^{-n-1}  \right) \, T \, dr \\
	&\leq 2n\, \frac{V(r_2)}{r_2^{n+2}}  - 4\, \frac{T(r_1)}{r_1^{n+2}}  + 2\, \int_{r_1}^{r_2} \left( 2 \, (n+2)     -  r^2  \right) \, \frac{T}{r^{n+3}} \, dr \, , \notag
\end{align}
where the last inequality used that $T(r) \leq \frac{n}{2} \, V(r)$. Since $r_1^2 \geq 2n + 4$, the last term on the right is non-positive, so we get the claim \eqr{e:evg} when $\Sigma$ is an immersed shrinker.
By standard arguments (sections 17--18 in \cite{S}), this extends to $F$-stationary $n$-dimensional
rectifiable varifolds in $\RR^N$.
\end{proof}

\subsection{Singularities}

An   $F$-stationary rectifiable varifold $\Sigma$  can be decomposed into  an open set $\Reg (\Sigma)$ of regular points,
where  $\Sigma$ is an embedded submanifold,  together with a closed singular set $\Sing (\Sigma) = \Sigma \setminus \Reg (\Sigma)$.

 If $\Sigma^n \subset \RR^{n+1}$ is  $F$-minimizing, then it is also area-minimizing for the  Gaussian metric (see, e.g., page 770 in \cite{CM1}).  Therefore, the 
 regularity theory for area-minimizers (Theorem $4$ in \cite{SS} or section $18$ in \cite{Wi}) implies that 
\begin{align}
	\dim_H \Sing (\Sigma) \leq n-7 \, ,
\end{align}
where  $	\dim_H$  is Hausdorff dimension.  By Theorem $5.8$ in \cite{CN}, the same bound holds for Minkowski
dimension $\dim_M$.
Recall that the $\rho$-covering number $\cov_{\rho} (S)$ of a set $S$ is the number of closed balls of radius $\rho > 0$ in a minimal covering{\footnote{Every ball in the covering has the same radius $\rho$; this differs from coverings used to compute Hausdorff measure and dimension, where the balls can have different radii.}}
 of $S$.
A set 
  $S \subset \RR^N$ has Minkowski dimension less than $k$ if 
  \begin{align}	\label{e:minkd}
	\lim_{\rho \to 0} \, \rho^{k} \, \cov_{\rho} (S) = 0 \, ,
\end{align}
 Since each ball has Euclidean volume comparable to $\rho^N$, an immediate consequence of \eqr{e:minkd} is that
  \begin{align}
	\lim_{\rho \to 0} \, \rho^{k-N} \, \Vol_{\RR^N} (T_{\rho} (S) ) = 0 \, ,
\end{align}
where $T_{\rho} (S)$ is the Euclidean tubular neighborhood of radius $\rho > 0$ of $S$.   
A bound on $\dim_M$ is stronger than a bound on the Hausdorff dimension $\dim_H$ and, in particular, 
   implies capacity estimates and removability of singularities for elliptic equations.

\vskip1mm
The next  lemma constructs good cutoff functions on the regular part
when the Minkowski dimension of the singular set is less than $n-2$.

\begin{Lem}	\label{l:singcut} 
Let $\Sigma^n \subset \RR^N$ be a proper $F$-stationary rectifiable varifold in $B_{R}$ with 
\begin{align}	\label{e:dmbd}
	\dim_M \Sing (\Sigma) < n-2 \, .
\end{align}
Given $\epsilon > 0$,   there is a function $\phi:\Sigma \to [0,1]$  that  is zero on $\overline{B_{R-1}} \cap \Sing (\Sigma)$ and
\begin{align}	\label{e:singcut}
 \int_{\Sigma} |\nabla \phi|^2 \, \e^{ - \frac{|x|^2}{4}} +  \int_{\Sigma} (1- \phi^2) \, \e^{ - \frac{|x|^2}{4}}  &\leq \epsilon  \, .
\end{align}
\end{Lem}

\begin{proof}
Set $S = \overline{B_{R-1}} \cap \Sing (\Sigma)$.  Let $\delta > 0$ be a small constant to be chosen below.  The assumption
\eqr{e:dmbd} implies that there exists $\rho < \frac{1}{3R}$ so that
\begin{align}
	\cov_{\rho} (S) < \delta \, \rho^{2-n} \, .
\end{align}
Let $\{ B_{\rho} (x_i)\}$, $i=1 , \dots , m$ with $m < \delta \, \rho^{2-n}$ be a minimal covering of $S$ by balls of radius $\rho$.  In particular, we have that every $x_i \in \overline{B_{R-1 + R^{-1}}}$.  Let $d_S (x)$ be the Euclidean distance from $x$ to $S$ and define the function $\phi$ by
\begin{align}
	\phi (x) = \begin{cases}
	0 &{\text{ if }} d_S(x) \leq \rho \\
	\rho^{-1}  \, d_{S} (x) - 1 &{\text{ if }} {\rho} < d_S(x) < {2}{\rho} \\
	1 &{\text{ if }}    2\, \rho \leq d_{S} (x) 
	\end{cases}
	\, .
\end{align}
Note that $\phi$ maps to $[0,1]$ and vanishes on $S$.  Moreover, both $(1-\phi^2)$ and $|\nabla \phi|$ have support in the tubular neighborhood $T_{2\rho}(S)$ of radius $2\rho$ about $S$.
Since $|\nabla \phi| \leq \rho^{-1}$, the exponential weight is at most one, and $T_{2\rho}(S)$ is contained in the union of the balls $B_{3\rho}(x_i)$,  it follows that
\begin{align}	 	\label{e:volTbd}
 \int  |\nabla \phi|^2 \, \e^{ - \frac{|x|^2}{4}} +  \int  (1- \phi^2) \, \e^{ - \frac{|x|^2}{4}}  &\leq 
 \left( 1 + \rho^{-2} \right) \, \Vol \left( T_{2\rho}(S) \cap \Sigma \right) \notag \\
 &\leq 2\, \rho^{-2} \, \sum_i \Vol \left( B_{3\rho} (x_i) \cap \Sigma \right) \\
 &< 2\, \delta \, \rho^{-n} \, \max_i \, \Vol \left( B_{3\rho} (x_i) \cap \Sigma \right) \, .  \notag
\end{align}
The shrinker equation implies that $\Sigma$ has mean curvature weakly in $L^{\infty}$ and bounded by $\frac{1}{2R}$ on $B_R$. Therefore, since $3\rho \leq \frac{1}{R}$ and 
$B_{\frac{1}{R}} (x_i) \subset B_R$ for every $i$, 
  monotonicity of volume with bounded mean curvature (Lemma $1.18$ in \cite{CM2}; cf. section $18$ in \cite{S}) gives a constant $c$ so that
  \begin{align}
  	\Vol \left( B_{3\rho} (x_i) \cap \Sigma \right) \leq c \, \left( \frac{ \rho}{ R^{-1} } \right)^n  \, \Vol (B_{ \frac{1}{R} } (x_i) \cap \Sigma) \leq
	c\, \rho^n \, R^n \, \Vol (B_R \cap \Sigma) \, .
  \end{align}
Using this in \eqr{e:volTbd} gives
\begin{align}	 	\label{e:volTbd2}
 \int  |\nabla \phi|^2 \, \e^{ - \frac{|x|^2}{4}} +  \int  (1- \phi^2) \, \e^{ - \frac{|x|^2}{4}}  & < 2\, c\, \delta  R^n \,  \Vol (B_R \cap \Sigma) \, .
 \end{align}
 Since $R$ here is fixed and $\Sigma$ is proper, the right-hand side of \eqr{e:volTbd2} is bounded by $C' \, \delta$ for a fixed constant $C'$ independent of $\delta$.   The lemma follows by choosing $\delta > 0$ sufficiently small depending on $\epsilon$ and $C'$.
\end{proof}

 \section{A strong Bernstein theorem for shrinkers}
 
  The classical Bernstein theorem for minimal surfaces gives that an entire minimal graph in $\RR^3$ must be flat.  This was extended to area-minimizing hypersurfaces up to $\RR^7$ and shown to be false in $\RR^8$    (see, e.g., \cite{S, CM2}).  
        The second variation of Gaussian area was introduced in 
 \cite{CM1}  and used to show that there are no complete Gaussian stable hypersurfaces{\footnote{This required that there is at most polynomial volume growth; by \cite{DX}, this holds when $\Sigma$ is proper. }}.  
Wang used the second variation and a calibration argument to prove a  Bernstein theorem for entire graphical shrinkers
  in 
 \cite{W1}.

   In this section, we will prove the strong Bernstein theorem, Theorem \ref{t:stable}.  In fact, we show the following more
   general theorem that allows boundary and a small set of singularities:
   
   \begin{Thm}	\label{t:stable2}
If $\Sigma^n \subset B_R \subset  \RR^N$ is a proper $F$-stationary rectifiable varifold with $R \geq R_n$  and $\dim_M \Sing (\Sigma) < (n-2)$, then $\Sigma$ is Gaussian unstable.
\end{Thm}

Theorem \ref{t:stable2} has
Theorem \ref{t:stable} as an immediate consequence when the singular set is empty.  We will need this more general version to prove 
the strong Frankel theorem in high dimensions where Gaussian minimizers can have singularities.

\subsection{Second variation and stability}

The second variation operator $L$ for Gaussian area from \cite{CM1} (equation ($4.13$) in \cite{CM1} for hypersurfaces, Theorem $4.1$ in \cite{AHW}, cf. \cite{AS,LL,CM5,CM6} for higher codimension) is given on normal vector fields $v$ by
\begin{align}
	L\,v = \cL\,v + \langle A_{ij},v\rangle\,A_{ij} + \frac{1}{2}\,v \, .
\end{align}
Here $\cL\,v=\Delta^{\perp}\,v-\frac{1}{2}\,\nabla^{\perp} _{x^T}v$, $\perp$ denotes the component normal to $\Sigma$, $A$ is the second fundamental form, and $A_{ij}=A(f_i,f_j)$, where $f_i$ is an orthonormal basis for the tangent space of $\Sigma$ at the given point.  
When $\Sigma$ is a hypersurface with unit normal $\nn$, then $L = \cL + |A|^2 + \frac{1}{2}$.

An $F$-stationary rectifiable varifold $\Sigma^n \subset \RR^N$  is said to be Gaussian stable if 
\begin{align}	\label{e:stabineq}
	  \int_{\Sigma} \langle \eta, \, L\, \eta \rangle\, \e^{ - \frac{|x|^2}{4}} \leq 0 
\end{align}
for every normal vector field $\eta$ with compact support on   $\Reg (\Sigma )$.  The inequality \eqr{e:stabineq} is known as the stability inequality.

\vskip1mm
The next proposition gives a uniform Poincar\'e inequality on the regular part  for Gaussian stable $F$-stationary rectifiable varifolds:

\begin{Pro}	\label{p:stable1}
If $\Sigma^n \subset \RR^N$ is a proper Gaussian stable $F$-stationary rectifiable varifold, then  every function $w$ with compact support
on $\Reg ( \Sigma )$ satisfies
\begin{align}	\label{e:12stab}
	\int_{\Sigma} w^2 \, \e^{ - \frac{|x|^2}{4}} \leq 2 \, \int_{\Sigma} |\nabla w|^2 \e^{ - \frac{|x|^2}{4}} \, .
\end{align}
\end{Pro}

\begin{proof}
When $\Sigma$ is a smooth embedded hypersurface, then \eqr{e:12stab} follows immediately from the stability inequality  with $\eta = w \, \nn$ (where $\nn$ is a unit normal) and integration by parts.  We will show that \eqr{e:12stab} holds in arbitrary codimension and without smoothness or embeddedness.

Let $e_1 , \dots , e_{N} \in \RR^N$ be a Euclidean orthonormal basis.  For each $i$, define  the normal vector field $v_i$ by taking the perpendicular part of the constant vector $e_i$
\begin{align}
	v_i = e_i^{\perp}\,  .
\end{align}
  Since the rank of the normal bundle is the codimension $N-n$, we have that
\begin{align}	\label{e:sumvi}
	\sum_{i} |v_i|^2 = N-n \, .
\end{align}
By Theorem $5.2$ in \cite{CM1} (see $5.1$ in \cite{AHW} for the higher codimension case), each $v_i$ 
satisfies{\footnote{Equation
\eqr{e:Lvi} is 
 special to shrinkers, as opposed to holding more generally for weighted minimal submanifolds in an ambient space with strictly positive Bakry-Emery Ricci curvature.}}
\begin{align}	\label{e:Lvi}
	L \, v_i = \frac{1}{2} \, v_i \, .
\end{align}
Now let $w$ have compact support on $\Reg ( \Sigma )$ and observe that
\begin{align}
	L \, (w \, v_i) = (\cL \, w) \, v_i + \frac{1}{2} \, w \, v_i + 2 \, \nabla^{\perp}_{\nabla w} v_i  \, .
\end{align}
Using $w\, v_i$ in the stability inequality \eqr{e:stabineq} gives
\begin{align}
	\frac{1}{2} \, \int w^2 \, |v_i|^2 \, \e^{ - \frac{|x|^2}{4}} \leq - \int_{\Sigma} 
	\left(  w\, (\cL \, w) \, |v_i|^2 + 2 \, w \, \langle v_i,\nabla^{\perp}_{\nabla w} v_i \rangle
	\right) \, \e^{ - \frac{|x|^2}{4}} \notag\\
	=- \int_{\Sigma} 
	\left(  w\, (\cL \, w) \, |v_i|^2 + w\,\langle \nabla w,\nabla |v_i|^2\rangle\right) \, \e^{ - \frac{|x|^2}{4}}\, .
\end{align}
Summing this over $i$ and using \eqr{e:sumvi} gives
\begin{align}
	\frac{N-n}{2} \, \int w^2  \, \e^{ - \frac{|x|^2}{4}} \leq - (N-n) \, \int_{\Sigma} 
	  w\, (\cL \, w)    \, \e^{ - \frac{|x|^2}{4}}  = (N-n) \, \int_{\Sigma} |\nabla w|^2 \, \e^{ - \frac{|x|^2}{4}} \, ,
\end{align}
where the last equality used the divergence theorem.
\end{proof}

\begin{proof}[Proof of Theorem \ref{t:stable2}]
Let $R > \sqrt{4+2\,n}$ be a radius to be chosen large enough below.     
  Since 
  \begin{align}
  	\cL \, |x|^2 = 2\,n - |x|^2
\end{align}
 by Lemma $3.20$ in \cite{CM1}, the maximum principle guarantees that $\Sigma$ intersects the closure of 
  $B_{\sqrt{2\,n}}$.
  
 Lemma \ref{l:areagrow} gives for $0< r_1 < r_2 < R$ that 
\begin{align}
	\left( 1 - \frac{2n}{r_2^2} \right) \, \frac{\Vol (B_{r_2} \cap \Sigma)}{r_2^n} \leq \frac{\Vol (B_{r_1} \cap \Sigma)}{r_1^n}  \, .
\end{align}
Let $\eta$ be a radial (i.e., a function of $|x|$) cutoff function on $\Sigma$ that is identically one on $B_{r_2-1}$ and cuts off linearly to zero from $r_2-1$ to $r_2$. 
Given $\epsilon > 0$,  Lemma \ref{l:singcut} gives a function
$\phi =\phi_{\epsilon}$ that vanishes on $B_{r_2} \cap \Sing (\Sigma)$, has $0 \leq \phi \leq 1$, and 
\begin{align}	\label{e:phiep1}
	 \int_{\Sigma} |\nabla \phi|^2 \, \e^{ - \frac{|x|^2}{4}} &\leq \epsilon \, , \\
	  \int_{ \Sigma}  (1- \phi^2) \, \e^{ - \frac{|x|^2}{4}} &\leq \epsilon \, . \label{e:phiep2}
\end{align}
It follows that $w = \eta \, \phi$ has compact support in the regular part of $B_{r_2} \cap \Sigma$.
Since   $\eta , \phi$ map to $[0,1]$,  the squared triangle inequality gives that
\begin{align}
	|\nabla w|^2 \leq 2\, |\nabla \eta|^2 + 2\, |\nabla \phi |^2 \, .
\end{align}
Integrating this and using that $|\nabla \eta| \leq 1$ vanishes inside $B_{r_2 -1}$ gives
\begin{align}
	\int |\nabla w|^2 \, \e^{ - \frac{|x|^2}{4} } &\leq 2\, \int   |\nabla \phi|^2 \, \e^{ - \frac{|x|^2}{4}} +
	 2\, \e^{ - \frac{ (r_2 -1)^2}{4}} \, \Vol (B_{r_2} \cap \Sigma) \notag \\
	&\leq 2\, \epsilon + 2\, \e^{ - \frac{ (r_2 -1)^2}{4}} \, \Vol (B_{r_2} \cap \Sigma) \, ,
\end{align}
where all integrals are on $\Sigma$ and the last inequality used \eqr{e:phiep1}. On the other hand, taking $r_2 > r_1 + 1$ so that $\eta \equiv 1$ on $B_{r_1}$, we have that
\begin{align}
	\int w^2 \, \e^{ - \frac{|x|^2}{4} } &\geq \int_{B_{r_1} \cap \Sigma} w^2 \, \e^{ - \frac{|x|^2}{4} } =
	\int_{ B_{r_1} \cap \Sigma} \phi^2 \, \e^{ - \frac{|x|^2}{4} } \notag \\
	&\geq
	\e^{ - \frac{r_1^2}{4} } \, \Vol (B_{r_1} \cap \Sigma) -  \int (1-\phi^2) \,  \e^{ - \frac{|x|^2}{4}} \geq 
	\e^{ - \frac{r_1^2}{4} } \, \Vol (B_{r_1} \cap \Sigma) - \epsilon \, ,
\end{align}
where the last inequality used \eqr{e:phiep2}.   If $\Sigma$ was Gaussian stable, then 
 Proposition \ref{p:stable1} would give that
 \begin{align}
 	\int w^2 \, \e^{ - \frac{|x|^2}{4}} \leq 2 \, \int |\nabla w|^2 \, \e^{ - \frac{|x|^2}{4} } \, .
 \end{align}
 Using the last two bounds on $w$ and $\nabla w$, this would imply that
 \begin{align}
 	\e^{ - \frac{r_1^2}{4} } \, \Vol (B_{r_1} \cap \Sigma) - \epsilon  \leq 4\, \epsilon + 4\, \e^{ - \frac{ (r_2 -1)^2}{4}} \, \Vol (B_{r_2} \cap \Sigma) \, .
 \end{align}
 Since $\epsilon > 0$ is arbitrary, this gives 
  \begin{align}
 	\e^{ - \frac{r_1^2}{4} } \, \Vol (B_{r_1} \cap \Sigma)   \leq 4\, \e^{ - \frac{ (r_2 -1)^2}{4}} \, \Vol (B_{r_2} \cap \Sigma) \, .
 \end{align}
 However, Proposition \ref{l:areagrow}
gives that
\begin{align}
	\left( 1 - \frac{2n}{r_2^2} \right) \, \frac{\Vol (B_{r_2} \cap \Sigma)}{r_2^n} \leq \frac{\Vol (B_{r_1} \cap \Sigma)}{r_1^n}  \, .
\end{align}
The last two inequalities are contradictory for $r_1$ fixed and $r_2$ large enough, so $\Sigma$ cannot be Gaussian stable.
\end{proof}

\subsection{Proof of the strong Frankel theorem}

   The idea for the strong Frankel theorem is that if there are two distinct shrinkers in a large ball, then 
a barrier argument produces an $F$-minimizer between them, violating the strong Bernstein theorem.

\begin{proof}[Proof of Theorem \ref{c:connecting}]
Applying the maximum principle to 
$|x|^2$, it follows that 
\begin{align}
	   \overline{B_{\sqrt{2n}}} \cap \Sigma_i \ne \emptyset {\text{ for }} i=1,2 \, .
\end{align}
Let $R_n$ be from Theorem \ref{t:stable2}.
 We will argue by contradiction, so assume that 
\begin{align}
	B_{R_n} \cap \Sigma_1 \cap \Sigma_2= \emptyset \, .
\end{align} 
    In particular, we can find a line segment 
    \begin{align}
    	\sigma \subset \overline{B_{\sqrt{2n}}}
\end{align}
 that starts at a point in $\Sigma_1$, ends at a point in $\Sigma_2$, and is otherwise disjoint from $\Sigma_1 \cup \Sigma_2$.  Let $\Omega$ denote the component of 
$B_{R_n} \setminus (\Sigma_1 \cup \Sigma_2)$ that contains the interior of $\sigma$.  It follows that $\sigma$ has non-zero linking number with $\partial \Sigma_1$ in $\Omega$.  

Next, we solve for the $F$-minimizer $\Gamma \subset \overline{\Omega}$ with $\partial \Gamma = \partial (B_{R_n} \cap \Sigma_1)$ (see, e.g., \cite{Fe} $5.1.6$). The part $B_{R_n} \cap \Gamma$ of $\Gamma$ disjoint from $\partial B_{R_n}$  must be $F$-stationary, but $\Gamma$ might contain pieces in $\partial B_{R_n}$ that do not satisfy the shrinker equation. However, 
 the linking argument implies that $\Gamma \cap \sigma \ne \emptyset$.  Thus, there is a connected component $\Gamma_1$ of $B_{R_n} \cap \Gamma$ with 
 \begin{align}
 	\Gamma_1 \cap \sigma \ne \emptyset \, .
\end{align}
  For $n \leq 6$,
the interior regularity theory for area-minimizers, Theorem $4$ in \cite{SS} or section $18$ in \cite{Wi}, 
 gives that $\Gamma_1$ is smooth, embedded and stable.  This contradicts Theorem \ref{t:stable}, completing the proof when $n \leq 6$.  
    For $n \geq 7$, $\Gamma_1$ can have singularities, but \cite{SS} and Theorem $5.8$ in \cite{CN} give that
 $\Reg (\Gamma_1)$ is stable and 
 \begin{align}
 	\dim_M \, \Sing (\Gamma_1) \leq n- 7 \, .
\end{align}
 This contradicts Theorem \ref{t:stable2}, so we conclude that 
 \begin{align}
 	B_{R_n} \cap \Sigma_1 \cap \Sigma_2 \ne \emptyset \, .
\end{align}
\end{proof}

  \begin{proof}[Proof of Corollary \ref{c:connectup}]
Connectedness follows immediately from Theorem \ref{c:connecting} by regarding the possible connected components of $B_R \cap \Sigma$ as distinct shrinkers with boundary in $\partial B_R$.
  \end{proof}
  
   \begin{proof}[Proof of Corollary \ref{c:halfspace}]
  This follows from Theorem \ref{c:connecting} with $\Sigma_1 = \Sigma$ and $\Sigma_2$ equal to a generalized cylinder $\SS^k_{\sqrt{2k}} \times \RR^{n-k}$.
  \end{proof}
  
   Theorem \ref{t:halfspace} follows immediately from the following stronger effective version:
  
  \begin{Thm}	\label{t:halfspace2}
Let $\Omega^+$ and $\Omega^-$ be the  components of the complement of 
$\SS^k_{\sqrt{2k}} \times \RR^{n-k}$ for some $k \in \{ 0,1, \dots , n\}$.
If $\Sigma^n   \subset \RR^{n+1}$ is a properly
embedded shrinker, $R \geq R_n$, 
\begin{align}
	\partial \Sigma  \subset \partial B_{R} , \,  B_{R} \cap \Sigma \ne \emptyset {\text{ and }} B_R \cap \Sigma  {\text{ intersects at most one of }} \Omega^{\pm} \, , 
\end{align}
then $B_{R} \cap \Sigma \subset \SS^k_{\sqrt{2k}} \times \RR^{n-k}$.
\end{Thm}

  \begin{proof} 
  By assumption, $B_R \cap \Sigma$ lies on one side of $\SS^k_{\sqrt{2k}} \times \RR^{n-k}$.  Since 
    Corollary \ref{c:halfspace} gives that $B_R \cap \Sigma$ is connected and  must intersect $\SS^k_{\sqrt{2k}} \times \RR^{n-k}$, the strong maximum principle (for minimal hypersurfaces in a Riemannian manifold) 
    gives that they must agree identically.
  \end{proof}

\end{document}